\documentclass{amsart}
\usepackage{amssymb}
\usepackage{graphicx}
\usepackage{amscd}
\usepackage{color}
\allowdisplaybreaks
\usepackage{tikz}
\makeatletter
\@namedef{subjclassname@2020}{\textup{2020} Mathematics Subject Classification}
\makeatother

\newcommand{\ignore}[1]{}

\numberwithin{figure}{section}
\numberwithin{table}{section}

\newcommand\tr{\operatorname{tr}}

\newcommand\vol{\mathsf{vol}}

\newcommand\R{\mathbb{R}}

\newcommand\D{d}

\newcommand\I{{\mathcal I}}

\newcommand\Lie{\mathcal L}
\renewcommand\P{{\mathcal P}}

\renewcommand\S{{\mathcal S}}

\newcommand{\0}{\mathaccent23}

\newcommand\Alt{\operatorname{Alt}}

\numberwithin{equation}{section}

\newtheorem{thm}{Theorem}[section]
\newtheorem{prop}[thm]{Proposition}
\newtheorem{lem}[thm]{Lemma}

\begin{document}

\parskip=10pt plus 2pt minus 2pt
\title[Cochain extensions by blending]
{Construction of polynomial preserving cochain extensions by blending}
\author{Richard S. Falk}
\address{Department of Mathematics,
Rutgers University, Piscataway, NJ 08854}
\email{falk@math.rutgers.edu}
\thanks{}
\author{Ragnar Winther}
\address{Department of Mathematics,
University of Oslo, 0316 Oslo, Norway}
\email{rwinther@math.uio.no}
\subjclass[2020]{Primary:  65N30, 65D17, 65D18}
\keywords{blending methods, simplices, differential forms, cochain extensions, 
preservation of polynomial spaces}
\date{February 3, 2022}
\thanks{The research leading to these results has received funding from the  European Research Council under the European Union's Seventh Framework Programme (FP7/2007-2013) / ERC grant agreement 339643. }

\begin{abstract}
  A classical technique to construct polynomial preserving extensions
  of scalar functions defined on the boundary of an $n$ simplex to the
  interior is to use so-called rational blending functions.  The
  purpose of this paper is to generalize the construction by blending
  to the de Rham complex.  More precisely, we define polynomial
  preserving extensions which map traces of  $k$ forms
  defined on the boundary of the simplex to $k$ forms defined in the interior.
  Furthermore, the extensions are cochain maps, i.e., they commute
  with the exterior derivative.
\end{abstract}

\maketitle

\section{Introduction}
\label{sec:intro}

In applications such as the finite element approximation of partial
differential equations and in computer aided geometrical design
problems, there often arises a need for a method for extending a
piecewise smooth function given on the boundary of a domain to the
entire domain, in particular when the domain is a hypercube or a
simplex.  There is a considerable literature on this subject, dating
back to early work of Coons \cite{Coons} in the context of computer
aided design, while more mathematical oriented studies of such
problems were initiated in \cite{Barnhill-Birkhoff-Gordon, Birkhoff}.
The constructions were often referred to as {\it transfinite
  interpolation} or {\it blending function methods}, since the
extension is obtained by combining, or blending, the transfinite
boundary data using rational or polynomial basis functions.  In
\cite{Mansfield2}, the two dimensional scheme described in
\cite{Barnhill-Birkhoff-Gordon} was generalized to the case of
tetrahedra, and with a brief discussion of its generalization to
$n$ simplices. Alternative approaches, using polynomial rather than
rational blending functions, were studied in \cite{Barnhill-Gregory,
  Gordon-Hall,Gregory2,Perronnet}.  A summary of much of this early work can be
found in \cite{Barnhill}.

More recently, the study of extension operators that preserve a
polynomial structure of the boundary data have played a key role in
the analysis of finite element methods of high polynomial order,
cf. \cite{Ainsworth-Demkowicz, MR899702,
  Bernardi-Dauge-Maday, Munoz-Sola}.  In particular, the importance of
such extensions that commute with the exterior derivative, i.e., {\it
  cochain extensions}, was illustrated by the analysis given in
\cite{MR2105164}.  The results of these papers further motivated the
theory developed in the series of three papers, \cite{DGS-extension3,
  DGS-extension1, DGS-extension2}, where polynomial preserving cochain
extensions are constructed for the de Rham complex in three
dimensions.  An important additional property of these extensions is
that they require only weak regularity of the boundary data to be well
defined.

The purpose of this paper is to extend the method of blending to
define polynomial preserving cochain extensions for differential forms
of arbitrary order on $n$ simplices.  More specifically, for
$\S_n = [x_0, x_1, \ldots , x_n] \subset \R^{n}$, an $n$ dimensional
simplex, we define extensions $E_n^k$ which map piecewise smooth $k$
forms defined on the boundary, $\partial \S_n$, to smooth forms on
$\S_n$, such that they preserve polynomial structures and commute with
the exterior derivative. For scalar-valued functions, or zero forms,
the extensions presented here correspond to operators defined in
\cite{Mansfield2}, while the general construction for higher order
forms appears to be new.

An outline of the paper is as follows. In Section~\ref{sec:prelim}, we
introduce some basic notation and recall the construction of
extensions by blending in the case of zero-forms, i.e., for
scalar-valued functions.  In particular, we verify that these
extensions are polynomial preserving.  We also present a summary of
the main results of the paper and an application of the construction.
In Section~\ref{sec:one-forms}, we discuss
the extension in the case of one-forms.  The explicit construction in
this basic case provides a motivation for the general construction for
$k$-forms to follow.  The discussion in Section~\ref{sec:one-forms}
also motivates the construction of a family of order-reduction
operators, presented in Section~\ref{sec:Rop}, which will play a key
role in designing the coefficient operators, $A_{I,J}^k$, that will be
used to define the extensions $E_n^k$.  A precise
  definition of the coefficient operators $A_{I,J}^k$ is given in
  Section~\ref{sec:AIJop}, and three key properties of these operators
  are established. In Section~\ref{sec:Tnk}, we then show that the
  operators $E_n^k$ are polynomial preserving cochain extensions.

\section{preliminaries}\label{sec:prelim}
\subsection{Notation}\label{sec:notation}
We will use $[\cdot, \ldots ,\cdot]$ to denote the simplex obtained by
convex combination of the arguments.  The simplex
$\S_n = [x_0, x_1, \ldots , x_n] \subset \R^{n}, \, n \ge 1,$ will be
considered to be a fixed $n$ simplex throughout this paper. The
barycentric coordinate associated to the vertex $x_i$ will be denoted by
$\lambda_i = \lambda_i(x)$, i.e., $\lambda_i$ is a linear function on
$\S_n$ satisfying $\lambda_i(x_j ) = \delta_{i,j}$ and such that
\[
x = \sum_{i=0}^n \lambda_i(x) x_i, \quad x \in \S_n.
\]
Furthermore, the boundary of $\S_n$, $\partial \S_n$, consists of all
$x \in \S_n$ such that $\lambda_i(x) = 0$ for at least one index
$i \in \{0, \ldots ,n\}$.  We will use $\Lambda^k(\partial \S_n)$ to
denote the space of piecewise smooth $k$ forms defined on
$\partial \S_n$. More precisely, each element
$u \in \Lambda^k(\partial \S_n)$ is smooth on each $(n-1)$ dimensional
subsimplex of $\partial \S_n$, and with single-valued traces at the
interfaces. Correspondingly, $\Lambda^k(\S_n)$ will denote the space
of smooth $k$ forms defined on $\S_n$  and
  $\0 \Lambda^k(\S_n)$ will denote the space of forms in
  $\Lambda^k(\S_n)$ with vanishing trace on $\partial \S_n$.
Our main goal is to
construct extension operators
$E_n^k : \Lambda^k(\partial \S_n) \to \Lambda^k(\S_n)$, $k = 0,1, \ldots, n-1$,
with desired properties.

We will use $d = d^k  : \Lambda^k(\S_n) \to 
\Lambda^{k+1}(\S_n)$ to denote the exterior derivative defined by 
\begin{equation*}
d u_x(v_1, \ldots, v_{k+1}) = \sum_{j=1}^{k+1} (-1)^{j+1} \partial_{v_j}
u_x(v_1, \ldots, \hat v_j, \ldots, v_{k+1}),
\end{equation*}
where the hat symbol, e.g., $\hat v_j$, is used to indicate a
suppressed argument and the vectors $v_j$ are elements of $\R^n$.  We
also use $u^1 \wedge u^2$ to denote the wedge product mapping a $j$
form $u^1$ and $k$ form $u^2$ into a $j+k$ form.
A smooth map $F: \S_n \to \partial \S_n$ 
provides a pullback of a differential form from $\partial \S_n$ to $\S_n$, 
i.e., a map from $\Lambda^k(\partial \S_n) \to \Lambda^k(\S_n)$ given by
\begin{equation*}
(F^* u)_x(v_1, \ldots, v_k) 
= u_{F(x)}(D F_x(v_1), \ldots,  D F_x(v_k)).
\end{equation*}
The pullback respects exterior products and differentiation, i.e.,
\begin{equation*}
F^*(u^1 \wedge u^2) = F^* u^1 \wedge F^* u^2, \qquad
F^*(d u) = d (F^* u).
\end{equation*}
The pullback of the inclusion map of $\partial \S_n$ into $\S_n$ is
the trace map, $\tr_{\partial S_n}$, and we note that the spaces
$\Lambda^k(\partial \S_n)$ are precisely defined so that
$\Lambda^k(\partial \S_n) = \tr_{\partial S_n} \Lambda^k(\S_n)$.

For each integer $r >0$, the polynomial subspaces $\P_r\Lambda^k(\S_n)$
consist of all elements $u \in \Lambda^k(\S_n)$ such that for fixed
vectors $v_1, \ldots ,v_k$, the function $u_x(v_1, \ldots ,v_k)$, as a
function of $x$, is an element of $\P_r(\S_n)$, i.e., the space of
polynomials of degree less than or equal to $r$ defined on $\S_n$. The trimmed space
$\P_r^-\Lambda^k(\S_n)$ is the subspace consisting of all
$u \in \P_r\Lambda^k(\S_n)$ such that
$u \lrcorner (x-a) \in \P_r\Lambda^{k-1}(\S_n)$ for any fixed
$a \in \R^n$.  Here, the symbol  $\lrcorner$ is used to denote contraction, i.e., 
\[
(u \lrcorner (x-a))_x(v_1, \ldots , v_{k-1}) = u_x(x-a,v_1, \ldots , v_{k-1}).
\]
The corresponding spaces on the boundary are defined by
\[
\P_r\Lambda^k(\partial \S_n) = \tr_{\partial S_n} \P_r\Lambda^k(\S_n) 
\quad \text{and} \quad 
\P_r^-\Lambda^k(\partial \S_n) = \tr_{\partial S_n} \P_r^-\Lambda^k(\S_n).
\]
We refer to \cite{FEEC-book, acta, bulletin} for more details on the
polynomial and piecewise polynomial spaces of differential forms.

If $J= \{j_0, j_1, \ldots ,j_m\}$ is an ordered subset of
  $\{0,1, \ldots ,n \}$, but not necessarily increasingly ordered, we
  will refer to $J$ as an index set. We will use $\Gamma$ to denote
the set of all increasingly ordered subsets of $\{0,1, \ldots ,n
\}$.
In other words, if $J \in \Gamma$, then $J$ is an index set of the
form
\begin{equation*}
J = \{j_0, j_1, \ldots ,j_m\}, \quad \text{where } 
0 \le j_0< j_1\ldots < j_m \le n.
\end{equation*}
The number of elements in $J$ will be denoted $|J|$, and $\Gamma_m$ is
the subset of $\Gamma$ consisting of all $J \in \Gamma$ with $|J| =
m+1$. Furthermore, for any $I \in \Gamma$, we let
\[
\Gamma_m(I) = \{ J \in \Gamma_m \, : \, J \subset I \, \}.
\]
For an index set $J$, the complement of $J$ relative to
the set $\{0, 1,\ldots ,n\}$ will be denoted $J^c$ and 
$[x_J]$ will denote the simplex generated by the vertices
$\{ x_j \}_{j \in J}$, with orientation induced by the order of
$J$. Furthermore, $\lambda_J = \lambda_J(x)$ will denote the
corresponding sum of barycentric coordinates, i.e.,
$\lambda_J = \sum_{j \in J} \lambda_j$, and we will use $\phi_J$ to
denote the associated Whitney form, given by
\[
\phi_J = \sum_{i=0}^m (-1)^{i} \lambda_{j_i} d\lambda_{j_0} \wedge 
\ldots \wedge \widehat{d\lambda_{j_i} } \wedge
\ldots \wedge d\lambda_{j_m}, \quad \text{if } J =  \{j_0, j_1, \ldots ,j_m\}.
\]
We note that if $J = \{ j\} \in \Gamma_0$, then
$\phi_J = \lambda_J = \lambda_j$. Furthermore, the set
$\{ \phi_J \}_{J \in \Gamma_k}$ is a basis for
$\P_1^-\Lambda^k(\S_n)$.  We will also use the notation
\begin{equation*}
(\delta \phi)_J = \sum_{i =0}^m (-1)^i \phi_{J(\hat i)}, \qquad
J =  \{j_0, j_1, \ldots ,j_m\},
\end{equation*}
where $J(\hat i)$ refers to the set $J$, but with the index $j_i$ removed.

\subsection{Scalar-valued functions}\label{sec:zero-forms}
We next give a quick review of the extensions studied in
\cite{Mansfield2} for scalar-valued functions, or zero forms, but in a
slightly different notation.  For each index set $I \in \Gamma_m$,
$m \ge 1$, and $j \in \{0, \ldots ,n\}$, we define the map
$P_{I,j}: \S_n \to \partial \S_n$ by
\[
P_{I,j} x = x + \sum_{i \in I} \lambda_i(x) (x_j - x_i).
\]
In fact, $P_{I,j}$ is a projection onto the set
$\{ x \in \S_n : \lambda_i(x) = 0, \, i \in I, \, i \neq j \}$.  Note
that if $j \in I$, then $P_{I,j} = P_{I^\prime,j}$, where $I^\prime$
represents the index set obtained from $I$ by removing $j$. 
In Figure~\ref{projections}, we show all the possible points
$P_{I,j}x$, for $j \in I$, in the case when $n=2$.

\begin{figure}
\begin{center}
\begin{tikzpicture}[scale = 0.75]
\draw[-](0,0)--(5,0);
\draw[-](0,0)--(2.5,4);
\draw[-](5,0)--(2.5,4);
\draw[dashed](0.6,1)--(4.4,1);
\draw[dashed](1.4,0)--(3.1,3);
\draw[dashed](1.25,2)--(2.5,0);
%
%
\node[inner sep = 0pt,minimum size=6pt,fill=black!100,circle] (n2) at (5,0) {};
\node[inner sep = 0pt,minimum size=6pt,fill=black!100,circle] (n2) at (0,0) {};
\node[inner sep = 0pt,minimum size=6pt,fill=black!100,circle] (n2) at (2.5,4){};
%
%
%
\node at (-1.5,0) {$P_{[0,1,2],2} = x_2$};
\node at (6.6,0) {$x_0= P_{[0,1,2],0}$};
\node at (2.5,4.5) {$x_1= P_{[0,1,2],1}$};
\node at (1.9,1.5) {$x$};
\draw[dashed](0.6,1)--(4.4,1);
\draw[dashed](1.4,0)--(3.1,3);
\draw[dashed](1.25,2)--(2.5,0);

\node at (-0.5,1) {$P_{[0,2],2}x$};
\node at (5.5,1) {$P_{[0,2],0}x$};
\node at (0.25,2.2) {$P_{[0,1],1}x$};
\node at (4,3) {$P_{[1,2],1}x$};
\node at (1.2,-0.35) {$P_{[1,2],2}x$};
\node at (3.0,-0.35) {$P_{[0,1],0}x$};
\node[inner sep = 0pt,minimum size=6pt,fill=black!100,circle] (n2) at (1.95,1)
{};
\end{tikzpicture}
\caption{\label{projections} Projections $P_{I,j}x$ for $n=2$.}
\end{center}
\end{figure}
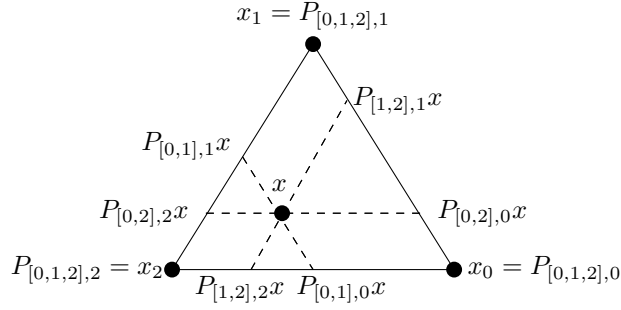

To define the extension
$E_n^0 : \Lambda^0(\partial \S_n) \to \Lambda^0(\S_n)$, we will utilize
the corresponding pullbacks $P_{I,j}^*$, defined by 
$(P_{I,j}^*u)_x = u(P_{I,j}x)$.  We define
\begin{equation*}
E_n^0u  = \frac{1}{n} \sum_{\substack{ I \in \Gamma_m \\ 1\le m \le n}} (-1)^{m+1} 
\sum_{j \in I}  \frac{\lambda_j}{\lambda_I} P_{I,j}^* u.
\end{equation*}
To see that $E_n^0$ is an extension, we need to check that
$\tr_{\partial \S_n} E_n^0 u = u$.  Without loss of generality,
consider a point $x \in \partial \S_n$ such that $\lambda_0(x) = 0$.
Then we have that
\[
 \frac{1}{n} \sum_{\substack{ I \in \Gamma_1 \\ 0 \in I}} 
\sum_{j \in I}  (\frac{\lambda_j}{\lambda_I} P_{I,j}^* u)_x = u_x.
\]
In addition, if $I \in \Gamma$ such that $0 \notin I$, and
$I^{\prime} = \{0,I\}$ then
\[
\Big(\frac{\lambda_j}{\lambda_I} P_{I,j}^* u\Big)_x 
= \Big(\frac{\lambda_j}{\lambda_{I^\prime}} P_{I^\prime,j}^* u\Big)_x.
\]
Therefore, the terms corresponding to $I$ and $I^\prime$ cancel at the
point $x$ and we can conclude that $(E_n^0 u)_x = u_x$. This shows
that {\it $E_n^0$ is an extension.}

Because of the rational factors in the definition of $E_n^0$, it is
not entirely obvious that the operator preserves the polynomial
structure of $u$. However, we observe that the pullbacks $P_{I,j}^*$
have the property that if
$u \in \P_r\Lambda^0(\partial \S_n)$, then
$P_{I,j}^*u \in \P_r\Lambda^0(\S_n)$. To show that the same property
holds for the extension $E_n^0$, we fix an $I \in \Gamma$ and consider
the sum $\sum_{j \in I} (\lambda_j/\lambda_I) P_{I,j}^* u$. If $I$ is
the maximal set $I = \{0, \ldots ,n\}$, then $\lambda_I(x) \equiv 1$ on
$\S_n$, so the sum is is simply $\sum_{j=0} ^n \lambda_j(x) u_{x_j}$,
which is the linear function that interpolates $u$ at the vertices.
For any other $I$, there is at least a vertex $x_p$ such
that $p \notin I$, and
\[
\sum_{j \in I} (\lambda_j/\lambda_I) 
P_{I,j}^* u = \sum_{j \in I} (\lambda_j/\lambda_I) 
(P_{I,j}^* u \pm P_{I,p}^* u) = P_{I,p}^* u + \sum_{j \in I} (\lambda_j/\lambda_I) 
(P_{I,j}^* u - P_{I,p}^* u).
\]
The first term on the right hand side is polynomial preserving, while we have 
\begin{multline}
\label{Pudiff}
P_{I,j}^* u_x -  P_{I,p}^*u_x = \int_0^1 \frac{d}{d\tau}
u_{(1-\tau)P_{I,p} x + \tau P_{I,j} x} \, d\tau \\
= \lambda_I \int_0^1 (d u)_{(1-\tau)P_{I,p} x + \tau P_{I,j} x}
 \lrcorner  (x_j - x_p) \,d\tau,
\end{multline}
which has $\lambda_I$ as a factor. Furthermore, the curve
$(1-\tau)P_{I,p} x + \tau P_{I,j} x$ for $\tau \in [0,1]$ belongs to
the set
$\{ x \in \S_n \, : \, \lambda_i(x) = 0, \, i \in I(\hat j) \}
\subset \partial \S_n$.
Therefore, we can conclude that
$ E_n^0(\P_r\Lambda^0(\partial \S_n)) \subset \P_r\Lambda^0(\S_n)$.

\subsection{The main result}\label{sec:main}
The discussion above shows that the operator
$E_n^0 : \Lambda^0(\partial \S_n)$ $\to \Lambda^0(\S_n)$ is a polynomial
preserving extension operator.  The main result of this paper is to
construct corresponding operators
$E_n^k :\Lambda^k(\partial \S_n) \to \Lambda^k(\S_n)$, $1\le k < n$,
satisfying $\tr_{\partial \S_n} E_n^k u = u$ and 
which are cochain extensions in the sense that the diagram
\begin{equation*}
\begin{CD}
\Lambda^0(\partial \S_n) @>d>> \Lambda^1(\partial \S_n)
@>d>> \ldots @>d>> \Lambda^{n-1}(\partial \S_n)
\\
@VV E_n^0 V @VV E_n^1 V @. @VV E_n^{n-1} V 
\\
\Lambda^0(\S_n) @>d>> \Lambda^1(\S_n)@>d>> 
\ldots @>d>>\Lambda^{n-1}(\S_n),
\end{CD}
\end{equation*}
commutes. In other words, $dE_n^{k-1} = E_n^k d$ for $1 \le k < n$,
and we will also show that $d E_n^{n-1} = 0$. Furthermore, the
extensions $E_n^k$ are polynomial preserving in the sense that
\begin{equation*}
E_n^k(\P_r\Lambda^k(\partial \S_n)) \subset \P_r\Lambda^k(\S_n) \quad \text{and }
E_n^k(\P_r^-\Lambda^k(\partial \S_n)) \subset \P_r^-\Lambda^k(\S_n), 
\end{equation*}
for $0 \le k < n$ and $r \ge 1$. 

All the operators $E_n^k$, $0 \le k <n,$ that we construct
will be rational functions of the form
\begin{equation}
\label{def-Tk}
E_n^ku  = \frac{1}{n} \sum_{\substack{I \in \Gamma_m\\ 1 \le m\le n}} (-1)^{m+1} 
\sum_{\substack{J \in \Gamma_s(I) \\ 0 \le s \le k}}
\frac{\phi_J}{\lambda_I^{s+1}} \wedge A^k_{I,J} u, 
\end{equation}
where the coefficient operators $A_{I,J}^k$ map
$\Lambda^k(\partial \S_n)$ to $\Lambda^{k-s}(\S_n)$ for
$J \in \Gamma_s(I)$, and such that $A_{I,J}^k = P_{I,j}^*$ for
$J = \{ j \} \in \Gamma_0$.
However, in contrast to what was done in the series of papers
\cite{DGS-extension3, DGS-extension1, DGS-extension2}, we will not
perform a careful discussion of bounds for the corresponding operator
norms.  It is clear that the extensions constructed by blending will
not be well defined on spaces with weak regularity, such as
$L^2(\partial \S_n)$. On the other hand, a lesson to be learned from
the series mentioned above, see also \cite{bubble-I, bubble-II}, is
that some additional averaging technique has to be used to be able
to construct such weak regularity extensions.

\subsection{Polynomial complexes}
\label{sec:poly-complex}
As an example of a direct application of the extensions $E_n^k$
constructed below, we will briefly consider polynomial complexes of
differential forms on $\R^n$, cf. \cite[Section 3.5]{acta} and
\cite[Section 4.2]{cost-mc}. For any $r>0$, the full polynomial
complex
\[                                                                              
\begin{CD}                                                                      
\R\hookrightarrow \P_r\Lambda^0 @>\D>> \P_{r-1}\Lambda^{1}                      
 @>\D>> \cdots @>\D>> \P_{r-n}\Lambda^n \to 0,                                  
\end{CD}                                                                        
\]
and the trimmed polynomial complex
\[                                                                              
\begin{CD}                                                                      
\R\hookrightarrow \P_r^-\Lambda^0 @>\D>> \P_r^-\Lambda^{1}                      
 @>\D>> \cdots @>\D>> \P_r^-\Lambda^n \to 0,                                    
\end{CD}                                                                        
\]
are both exact. In fact, by combining the full polynomial spaces and
the trimmed polynomial spaces, we can obtain $2^{n-1}$ different
exact polynomial complexes of the form
\begin{equation}\label{pol-exact}                                               
\begin{CD}                                                                      
\R\hookrightarrow \P\Lambda^0 @>\D>> \P\Lambda^{1}                              
 @>\D>> \cdots @>\D>> \P\Lambda^n \to 0,                                        
\end{CD}                                                                        
\end{equation}
where $\P\Lambda^0= \P_r\Lambda^0 = \P_r^-\Lambda^0$ and
$\P_{r-k}\Lambda^k \subset \P\Lambda^k \subset \P_r^-\Lambda^k$. In
particular, each space $\P\Lambda^k$ is either a trimmed or a full
polynomial space.  Furthermore, the exactness of these complexes can
be established by a standard Poincar\'e-type operator of the form
\[                                                                              
(Q^k u)_x = \int_0^1 \tau^{k-1} u_{(1- \tau)a + \tau x} \lrcorner
(x-a) \, d\tau,
\]
where $a \in \R^n$ is fixed. More precisely, if the spaces
$\P\Lambda^{k-1}$ and $\P\Lambda^{k}$ are related as in
\eqref{pol-exact}, then the operator $Q^k$ maps $\P\Lambda^{k}$ to
$\P\Lambda^{k-1}$, and if $u \in \P\Lambda^{k}$ satisfies $du = 0$,
then $u = dQ^ku$. In fact, any element $u$ of the polynomial spaces
$\P\Lambda^k$ admits the representation
\[                                                                              
u = dQ^ku + Q^{k+1} du.                                                         
\]
It also well-known that the corresponding polynomial complexes with
boundary conditions are exact. More precisely, we consider complexes
of the form
\[                                                                              
\begin{CD}                                                                      
\0\P\Lambda^0(\S_n) @>\D>> \0\P\Lambda^{1}(\S_n)                                
 @>\D>> \cdots @>\D>> \0\P\Lambda^n(\S_n) \to \R,                               
\end{CD}                                                                        
\]
where the spaces $\0\P\Lambda^{k}(\S_n)$ are the restriction of the
spaces $\P\Lambda^{k}$ of \eqref{pol-exact} to $\S_n$, and with
vanishing traces on $\partial S_n$.  The exactness can in this case be
established by counting degrees of freedom, using the results of
\cite[Section 4]{acta} or \cite[Section 5]{bulletin}.  On the other
hand, in this setting it is not straightforward to construct a simple
operator $\0 Q^k$ which maps $\0\P\Lambda^{k}(\S_n)$ to
$\0\P\Lambda^{k-1}(\S_n)$, and which has the property that
$u = d\0 Q^ku$ if $du=0$. For example, if we consider an operator of
the form $Q^k$ above, then it seems impossible to choose the point
$a \in \S_n$ such that the vanishing trace condition is preserved.
However, by using the extension operator $E_n^k$ constructed in this
paper, we can define $\0 Q^k$ as
\[                                                                              
\0 Q^k u = (I - E_n^{k-1} \circ \tr_{\partial S}) Q^k u,                        
\]
where the operator $Q^k$ is a Poincar\'e-type operator of the form
above, with $a \in \S_n$.  It follows from the properties of
$E_n^{k-1}$ and $Q^k$ that the operator $\0 Q^k$ maps
$\0\P_r\Lambda^k(\S_n)$ to $\0\P_{r+1}^-\Lambda^{k-1}(\S_n)$ and
$\0\P_r^-\Lambda^k(\S_n)$ to $\0\P_{r}\Lambda^{k-1}(\S_n)$.
Furthermore, we have
\[                                                                              
d \0 Q^ku +  \0 Q^{k+1} du = u - E_n^k\circ \tr_{\partial \S_n} u = u, \quad u 
\in \0\Lambda^k(\S_n).                                                          
\]

\section{The case of one forms}
\label{sec:one-forms}
To motivate the general construction for $k$ forms, we first consider
the construction for one forms. The operator $E_n^1$ must satisfy 
the commuting
property, $E_n^1 du = d E_n^0 u$. The right hand side of this identity is known,
and given by 
\begin{equation*}
d E_n^0 u  = \frac{1}{n} \sum_{\substack{I \in \Gamma_m\\ 1 \le  m \le n}} (-1)^{m+1} 
\sum_{j \in I} \Big[ d \Big(\frac{\lambda_j}{\lambda_I}\Big) P_{I,j}^* u
+ \frac{\lambda_j}{\lambda_I}  P_{I,j}^* du \Big],
\end{equation*}
where we have used the fact that $P_{I,j}^*$ commutes with the
exterior derivative.  The goal is to write the complete right hand
side in terms of $du$.  For a fixed $j \in I$, we have
\begin{equation*}
d \Big(\frac{\lambda_j}{\lambda_I}\Big)
= \frac{1}{\lambda_I^2} [\lambda_I d \lambda_j - \lambda_j d \lambda_I]
=\frac{1}{\lambda_I^2} \sum_{i \in I, i \neq j} 
[\lambda_i d \lambda_j - \lambda_j d \lambda_i]
= \sum_{i \in I, i \neq j} \frac{\phi_{i,j}}{\lambda_I^2}.
\end{equation*}
Since $\phi_{j,i} = - \phi_{i,j}$, we then obtain
\begin{equation*}
\sum_{j \in I}  d \Big(\frac{\lambda_j}{\lambda_I}\Big) P_{I,j}^* u
= \sum_{J \in \Gamma_1(I)}
\frac{\phi_J}{\lambda_I^2} (\delta P_I^* u)_{J},
\end{equation*}
where $(\delta P_I^* u)_{J} = P_{I,j}^*u - P_{I,i}^*u$ if $J = \{i,j \}$. 
Furthermore, as in \eqref{Pudiff}, we obtain 
\[
((\delta P_I^* u)_{J})_x = \lambda_I(x) \int_0^1 
(du)_{(1-\tau)P_{I,i} x + \tau P_{I,j} x} \lrcorner (x_j - x_i) \, d\tau, \quad 
J = \{ i,j \}.
\]
Since the above formula depends on $du$, rather than $u$, this leads
to a possible definition of $E_n^1 u$ such that the desired commuting
relation holds. More precisely, we can define
\begin{equation*}
E_n^1u = \frac{1}{n} \sum_{\substack{I \in \Gamma_m\\ 1 \le
m\le n}} (-1)^{m+1} \Big[ \sum_{j \in I}
\frac{\lambda_j}{\lambda_I} P_{I,j}^* u + \sum_{J \in \Gamma_1(I)}
\frac{\phi_J}{\lambda_I^2} R_{I,J}^1 u
\Big],
\end{equation*}
where the operator $R_{I,J}^1$ is defined by 
\[
(R_{I,J}^1 u )_x = \lambda_I(x) \int_0^1 u_{(1-\tau)P_{I,i} x + \tau P_{I,j} x} 
\lrcorner (x_j - x_i) \, d\tau
\]
for any $I \in \Gamma$ and $J$ an 
index set of length $2$.  Alternatively, we have
\[
(R_{I,J}^1 u )_x = \int_{P_{I,i} x}^{P_{I,j} x} u, \quad J = \{i,j \},
\]
where we have used differential form notation for writing the
integral of a one-form $u$ over the one-dimensional space 
$[P_{I,i}x, P_{I,j}x]$.
The problem with the definition of $E_n^1$ above is that the line
$[P_{I,i}x, P_{I,j}x]$ will in general not belong to the boundary
$\partial \S_n$, cf. Figure~\ref{projections}. We will therefore
replace the operator $R_{I,J}^1$ by an alternative operator,
$A_{I,J}^1$, given by
\begin{equation}
\label{A1Jij}
A_{I,J}^1 u= \frac{1}{n-1} \sum_{p \notin J} [R_{I,\{p,j\}}^1u - R_{I,\{p,i\}}^1u], \quad J = \{ i, j\} .
\end{equation}
This operator will still satisfy the key relation
$A_{I,J}^1 du = (\delta P_I^*u)_{J}$, and the line from
$[P_{I,j}x, P_{I,p}x]$ will belong to the set $\lambda_i = 0$ if
$J = \{i,j \}$ and $p \notin J$.  As a consequence, the operator
$A_{I,J}^1$ maps one forms defined on $\partial \S_n$ to scalar functions
defined on $\S_n$, and the operator $E_n^1$, defined by
\begin{equation}
\label{E1-formula} E_n^1u = \frac{1}{n} \sum_{\substack{I \in \Gamma_m\\ 1 \le
m \le n}} (-1)^{m+1} \Big[ \sum_{j \in I}
\frac{\lambda_j}{\lambda_I} P_{I,j}^* u + \sum_{J \in \Gamma_1(I)}
\frac{\phi_J}{\lambda_I^2} A_{I,J}^1 u\Big],
\end{equation}
will satisfy the commuting relation $dE_n^0u = E_n^1 du$.

To see that the operator $E_n^1$, defined by \eqref{E1-formula}, is an
extension, we have to show that $\tr_{\partial \S_n} E_n^1 u = u$.
This follows by essentially the same argument as for zero forms,
(cf. also the proof of Theorem~\ref{thm:extension} below).
As in the case of zero forms, we will have 
\[
 \tr_{\lambda_0 = 0}\frac{1}{n} \sum_{\substack{ I \in \Gamma_m \\ 1 \le m \le n}} 
(-1)^{m+1} \sum_{j \in I}  \frac{\lambda_j}{\lambda_I} P_{I,j}^* u 
=  \tr_{\lambda_0 = 0} u.
\]

To complete the argument, it will therefore be enough to show that
\begin{equation}\label{trace-prop-1}
 \tr_{\lambda_0 = 0} \Big[ \sum_{\substack{I \in \Gamma_m\\ 1 \le
m \le n}} (-1)^{m+1} \sum_{J \in \Gamma_1(I)}
\frac{\phi_J}{\lambda_I^2} A_{I,J}^1 u \Big] = 0.
\end{equation}
If $0 \in J$, then $\tr_{\lambda_0 = 0} \phi_J =0$. On the other hand,
if $0 \notin J$, then we can conclude from \eqref{A1Jij}
that for any $I \in \Gamma$ such that $0 \notin I$ and $J \in \Gamma_1(I)$,
\[
\tr_{\lambda_0 = 0} (A_{I,J}^1 u - A_{I^\prime,J}^1 u) =0, \quad I^\prime = \{ 0, I \}.
\]
Therefore, the terms corresponding to $I$ and $I^\prime$ on the left
hand side of \eqref{trace-prop-1} cancel. We can therefore conclude
that the identity \eqref{trace-prop-1} holds, and as a consequence,
$E_n^1$ is an extension.

The fact that the operator $E_n^1$ maps piecewise smooth one-forms defined
on $\partial \S_n$ to smooth one-forms defined on $\S_n$ and also preserves
the polynomial structure is not at all obvious, since the operator
$R_{I,J}^k$ and hence the operator $A_{I,J}^k$ has one factor of
$\lambda_I$, but not two.  However, we will show in Section~\ref{sec:Tnk},
using an alternative representation of the
operator $E_n^1$, that $E_n^1$ does, in fact, have these properties.

\section{The operators $R_{I,J}^k$ and their properties}
\label{sec:Rop}
To develop a formula for $E_n^k$ of the form \eqref{def-Tk}, in the
general case, we will first consider how to generalize the operators
$R_{I,J}^k$, introduced above, to the case when $k >1$ and $J$ is a
more general index set.  We recall that when $J$ is an index set, the
set $[x_J]$ is the simplex generated by the vertices
$\{ x_j \}_{j \in J}$. For any $x \in \S_n$ and index set
$I \in \Gamma$, the corresponding simplex $[P_{I,J}x]$ is the convex
combinations of the points $\{ P_{I,j}x \}_{j \in J}$. In general, the
simplex $[P_{I,J}x]$ will not be a subset of $\partial \S_n$,
cf. Figure~\ref{projections}.  However, if $I \cap J^c$ is nonempty,
where we recall that $J^c$ is the complement of $J$, we will indeed
have $[P_{I,J}x] \subset \partial \S_n$, since any $\lambda_i$, with
$i \in I \cap J^c$, will be identically zero on $[P_{I,J}x]$.

Key tools for the construction of the operators $R_{I,J}^k$ are the
maps $F_I : \S_n \times \partial S_n \to \S_n$ given by
\[
F_I(x,y) = x + \sum_{i \in I}\lambda_i(x) (y- x_i),
\]
where $I \in \Gamma$. We observe that $F_I(x,x_j) =
P_{I,j}x$.
Furthermore, if we restrict the domain of the map $F_I$ to
$\S_n \times [x_J]$, where $I \cap J^c$ is nonempty, then the range is
a subset of $\partial \S_n$. The pullback $F_{I}^*$, is a map
\[
F_{I}^* : \Lambda^k(\S_n) \to \Lambda^k(\S_n \times \partial \S_n).
\]
In the discussion below, we will be interested in operators of the form
$\tr_{\S_n \times [x_J]} \circ F_I^*$, mapping $\Lambda^k(\S_n)$ to
$\Lambda^k(\S_n \times [x_J])$. In particular, it follows from the
discussion above that if $I \cap J^c$ is nonempty, then this operator
maps $\Lambda^k(\partial \S_n)$ to $\Lambda^k(\S_n \times [x_J])$.

A space of $k$--forms on
a product space can be expressed using the tensor product  $\otimes$ as
\[
\Lambda^k(\S_n \times \partial \S_n) 
= \sum_{s=0}^k \Lambda^{k-s}(\S_n) \otimes \Lambda^{s}(\partial \S_n).
\]
In other words, elements 
$U \in \Lambda^{k-s}(\S_n) \otimes \Lambda^{s}(\partial \S_n)$ 
can be written as a sum of terms of the form  
\[
a(x,y) dx^{k-s} \otimes dy^{s},
\]
where $dx^{k-s}$ and $dy^{s}$ run over bases in $\Alt^{k-s}(\S_n)$ and
$\Alt^{s}(\partial \S_n)$, respectively, and where $a$ is a scalar
function on $\S_n \times \partial \S_n$.  
Here $\Alt^k$ is the corresponding space of algebraic $k$ forms.
Furthermore, for each $s$,
$0 \le s \le k$, there is a canonical map
$\Pi_s : \Lambda^k(\S_n \times \partial \S_n) \to \Lambda^{k-s}(\S_n)
\otimes \Lambda^{s}( \partial \S_n)$ such that
\[
U = \sum_{s=0}^k \Pi_s U, \quad U \in \Lambda^k(\S_n \times  \partial \S_n).
\]
The functions $\Pi_s F_{I}^* u \in \Lambda^{k-s}(\S_n) \otimes
\Lambda^{s}(\partial \S_n)$ can be identified as
\begin{multline*}
(\Pi_s F_{I}^* u)_{x,y} (v_1, \ldots, v_{k-s},  t_1, \ldots, t_{s})
\\
= u_{F_{I}(x,y)}(D_{x} F_{I} v_1,, \ldots, D_{x} F_{I} v_{k-s}, 
D_y F_{I} t_1, \ldots, D_{y} F_{I} t_{s}),
\end{multline*}
where the tangent vectors $v_i \in T(\S_n)$ and
$t_i \in T_y(\partial \S_n)$.  For the special function $F_I$ in our
case, we have $D_{y} F_{I} = \lambda_I(x) I$, while
\[
D_{x}F_{I} = D_xx + \sum_{i \in I}(y- x_i)d_x\lambda_i = \sum_{\ell \in I^c} (x_\ell -y) d_x \lambda_\ell,
\]
where $D_xx$ is the identity matrix.

The basic commuting property for pull-backs, namely $d F^* = F^* d$, can be
expressed in the present setting as
\begin{equation}\label{d-F*}
\Pi_s F_{I}^*du = \Pi_s dF_{I}^* u = d_{\S} \Pi_s F_{I,J}^*u 
-  (-1)^{k-s} d_{\partial \S}\Pi_{s-1} F_{I}^*u, \quad u \in \Lambda^k(\S_n),
\end{equation}
where $0 \le s \le k+1$, and where $d_{\S}$ and $d_{\partial \S}$
denote the exterior derivative with respect to the spaces $\S_n$ and
$\partial \S_n$, respectively.

Let $I \in \Gamma$ and $J$ an index set with $|J| = s+1$,
$0 \le s \le k \le n$.  We introduce a family of
operators $R_{I,J}^k$, mapping $\Lambda^k(\S_n)$ to
$\Lambda^{k-s}(\S_n)$, defined by
\begin{equation}
\label{Rdef}
(R_{I,J}^k u)_x = \int_{[x_J]} (\Pi_s F_{I}^* u)_x.
\end{equation}
For $s >k$, we define $R_{I,J}^k$ to be zero.
If $v_1, \ldots , v_{k-s}$ are vectors in 
$\R^{n+1}$ and $t_1, \ldots, t_s$ is  any orthonormal basis for the
tangent space $T[x_J] = T[P_{I,J} x]$, then
\begin{multline*}
(R_{I,J}^k u)_x(v_1, \ldots, v_{k-s})
=  \lambda_I(x)^s \int_{[x_J]} u_{F_{I}(x,y)} \lrcorner D_x F_{I} v_1 \lrcorner
\ldots \lrcorner D_x F_{I} v_{k-s} 
\\
= \lambda_I(x)^s \int_{[x_J]} (u_{F_{I}(x,y)} \lrcorner D_x F_{I} v_1 \lrcorner
\ldots \lrcorner D_x F_{I} v_{k-s}) ( t_1, \ldots 
, t_s )\, dy.
\end{multline*}
Note that since $u$ is a $k$-form, $u_{F_{I}(x,\cdot)} \lrcorner D_x F_{I} v_1 
\lrcorner \ldots \lrcorner D_x F_{I} v_{k-s}$
is an $s$ form with respect to $y$,
which we can then integrate over the $s$ dimensional space $[x_J]$.
In the final formula, we see that the integral to be evaluated is an
integral over $[x_J]$ for a fixed  $x$ and vectors $v_1, \ldots, v_{k-s}$.

In the special case when $s=0$, i.e., when the simplex $[x_J]$
degenerates to a vertex $x_j$, the operator $R_{I,J}^k$ is interpreted
as $P_{I,j}^*$.  When $s=1$ and $J = \{i,j \}$, we can
utilize the parameterization
$y = (1- \tau)x_i + \tau x_j, \, \tau \in [0,1]$ of $[x_J]$ to verify
that the definition of $R_{I,J}^1$ given in Section~\ref{sec:one-forms}
corresponds exactly to the definition given by \eqref{Rdef}.
If $I$ and $J$ are related such that $I \cap J^c$ is 
nonempty, then $R_{I,J}^k u$ will only depend on $\tr_{\partial \S_n} u$. 
Furthermore, if $i \in I \cap J^c$, then all the vectors 
$D_xF_I v_1, \ldots ,D_xF_I v_{k-s}$ and $t_1, \ldots, t_s$ will belong 
to the tangent space of the boundary simplex 
$\{ x \in \S_n \, : \, \lambda_i(x) = 0\}$, a space of dimension $< n$.
Hence, for $k=n$, it follows that all operators of the form
$R_{I,J}^n$ are identically zero in this case.

We will summarize the key properties of the operator $R_{I,J}^k$ in
the three lemmas given below.
\begin{lem}
\label{lem:Rtrace}
Let $I \in \Gamma$ and assume that $j \in \{0, \ldots ,n\}$ is such
that $j \notin I$. Then for all index sets $J$,
\begin{equation*}
\tr_{\lambda_j =0} R_{I,J}^k u = \tr_{\lambda_j =0} R_{I^\prime,J}^k u,
\end{equation*}
where $I^\prime \in \Gamma$ is equal to $\{j,I\}$ up to a
reordering. Moreover, if $I \cap J^c$ is nonempty, then $ R_{I,J}^ku$
only depends on $\tr_{\partial \S} u$, and the operators $R_{I,J}^n$
are all identically zero.
\end{lem}
\begin{proof}
  The properties obtained when $I \cap J^c$ is nonempty are already
  observed above.  Furthermore, we have for all $y \in \partial \S_n$
\[
F_I(x,y) \equiv  F_{I^\prime}(x,y), \quad 
\{ x  \in \S_n \, : \, \lambda_j(x) =0\}.
\]
As a consequence,
\[
\tr_{\lambda_j =0} F_I(\cdot ,y)^*u =  \tr_{\lambda_j =0} F_{I\prime}(\cdot ,y)^*u, 
\quad   y \in \partial \S_n.
\]
The desired result follows directly from the definition of
the operators $R_{I,J}^k$.
%
\end{proof}

It follows directly from the definition of the operators $R_{I,J}^k$
and the smoothness of the functions $F_I(\cdot,y)$ on $\S_n$ for any
fixed $y \in \partial \S_n$, that the operator
$\lambda_I^{-s} R_{I,J}^k$ maps $\Lambda^k(\S_n)$ to
$\Lambda^{k-s}(\S_n)$.  The corresponding result in the polynomial
case is given below.

\begin{lem}
\label{lem:Rprop1}
Let $I \in \Gamma$ and $J$ an index set with $|J|= s+1$.
\begin{itemize}
\item[i) ]If $u \in \P_r\Lambda^k(\S_n)$, then $\lambda_I^{-s}R_{I,J}^k u 
\in \P_r \Lambda^{k-s}(\S_n)$; 
\item[ii)] If $u \in \P_r^-\Lambda^k(\S_n)$ then 
$\lambda_I^{-s} R_{I,J}^k u \in \P_{r}^-\Lambda^{k-s}(\S_n)$.
\end{itemize}
Furthermore, if $I \cap J^c$ is nonempty, then the assumptions in the
two cases can be reduced to the trace conditions
$u \in \P_r\Lambda^k(\partial \S_n)$ and
$u \in \P_r^-\Lambda^k(\partial \S_n)$, respectively.
\end{lem}

\begin{proof}
Recall that 
\begin{multline*}
\lambda_I(x)^{-s}(R_{I,J}^k u)_x(v_1, \ldots, v_{k-s})
\\
=\int_{[x_J]} (u_{F_{I}(x,y)} \lrcorner D_x F_{I} v_1 \lrcorner
\ldots \lrcorner D_x F_{I} v_{k-s}) ( t_1, \ldots , t_s )\, dy,
\end{multline*}
where $t_1, \ldots, t_s$ is any orthonormal basis for the tangent
space $T[x_J]$.  Since $F_{I}$ is linear in $x$, the integrand
preserves the polynomial degree of $u$
for each fixed $y \in [x_J]$.
 Since the same will be true for the
integral with respect to $y$, the first part of the lemma follows.

To show the $\P_r^-$ spaces are also preserved, we look at 
$\lambda_I^{-s} R_{I,J}^ku \lrcorner (x - x_j)$, where $j \in J$.  
In fact, we choose $j \in I \cap J$ if this set is  nonempty.
It then follows that 
\[
DF_x(x-x_j) = \sum_{\ell \in I^c}(x_{\ell} -y) \lambda_{\ell}(x )= F_I(x,y) -y,
\]
which gives
\begin{equation}\label{R-contract}
(R_{I,J}^k u \lrcorner (x - x_j))_x = \int_{[x_J]}(\Pi_s F_{I}^* u^\prime)_x,
\end{equation}
where for each fixed $y$, we have
$u_x^\prime = u\lrcorner (x-y)$.  In other words,
$R_{I,J}^k u \lrcorner (x -x_j)$ satisfies the same definition as
$R_{I,J}^k u$, but with $u$ replaced by $u^\prime$.  However, if
$u \in \P_r^-\Lambda^k(\S_n)$, then $u^\prime \in \P_r\Lambda^k(\S_n)$
for each $y$, and hence the same argument as above shows that
$\lambda_I^{-s}R_{I,J}^k u \lrcorner (x -x_j) \in
\P_r\Lambda^k(\S_n)$.  Alternatively, if $I \cap J$ is empty, we obtain
\[
DF_x(x-x_j) = x- x_j + \sum_{i \in I} (y - x_i) \lambda_i(x) = F_I(x,y) - x_j.
\]
Also in this case, we obtain an expression of the form
\eqref{R-contract}, but now with $u^\prime 
= u\lrcorner (x-x_j)$. As
above, we again can conclude that
$\lambda_I^{-s} R_{I,J}^k u \lrcorner (x -x_j) \in
\P_r\Lambda^k(\S_n)$
if $u \in \P_r^-\Lambda^k(\S_n)$.  Hence, the second part of lemma has
been established. The final conclusion, related to the assumption
$I \cap J^c$ nonempty, again follows from the fact that $R_{I,J}^k u$
only depends on $\tr_{\partial \S_n}u$ in this case.
\end{proof}

For  $I \in \Gamma$ and index sets $J$  with $|J| = s+1$, we define 
\[
(\delta R_I^k u)_{J} = \sum_{i = 0}^s (-1)^i R_{I, J(\hat i)}^k.
\]
The following relation will be a key tool to show that the extensions
$E_n^k$ are cochain maps.
\begin{prop}
\label{prop:d-of-Q}
Let $I \in \Gamma$ and $J$ an index set with $|J|= s+1$, where
$0 \le s \le k+1$. The operators $R_{I,J}^k$ satisfy the relations
\begin{equation*}
d(R_{I,J}^ku) = R_{I,J}^{k+1} du + (-1)^{k-s}  (\delta R_I^k u)_{J}, 
\quad u \in \Lambda^k(\S).
\end{equation*}
\end{prop}

\begin{proof}
By applying \eqref{d-F*}, we get
\begin{align*}
dR_{I,J}^k u &= \int_{[x_J]} d_{\S} \Pi_s F_{I}^*u = \int_{[x_J]} [\Pi_s F_{I}^*du 
+ (-1)^{k-s} d_{[x_J]}\Pi_{s-1} F_{I}^*u]\\
&= R_{I,J}^{k+1} du + (-1)^{k-s} \int_{\partial [x_J]} \Pi_{s-1} F_{I}^* u,
\end{align*}
where we have used Stokes theorem for the last equality. 
The proof of the proposition is completed by observing that 
\[
\int_{\partial [x_J]} \Pi_{s-1} F_{I}^* u = (\delta R_I^ku)_{J}.
\]
\end{proof}

\section{The operators  $A_{I,J}^k$ and their properties}
\label{sec:AIJop}
In this section, we define the coefficient operators $A_{I,J}^k$ for
$I \in \Gamma_m$, $ 1 \le m \le n$, and $J \in \Gamma_s(I)$,
$0 \le s \le k \le n-1$, as operators mapping $\Lambda^k(\partial \S_n)$
to $\Lambda^{k-s}(\S_n)$.
 As we saw already in Section~\ref{sec:one-forms}, the
simplexes $[P_{I,J}x]$, defined as all convex combinations of the
points $P_{I,j} x$, $j \in J$, will in general not be a subset of
$\partial \S_n$, unless $I \cap J^c$ is nonempty. As a consequence,
the operators $R_{I,J}^k$ are not well defined for functions in $\Lambda^k(\partial \S_n)$ unless this condition
holds.  We will therefore define the operators $A_{I,J}^k$ as linear
combinations of these order reduction operators in such a way that they are well
defined for functions in $\Lambda^k(\partial \S_n)$.
We have already defined the operators $A_{I,J}^k$ for $J = \{j\} \in \Gamma_0(I) $ by
$A_{I,J}^k u =  R_{I,J}^k u =  P_{I,j}^*u$, and in the case $J \in \Gamma_1(I)$ by 
\eqref{A1Jij}. The general definition given below will generalize the definitions given in these special cases.
%

In the general case, we define the operators $A_{I,J}^k$ of the form 
\begin{equation}\label{def-Ak2}
 A_{I,J}^k =  c_{s,n}^k [(n-s) R_{I,J}^k  
- \sum_{p \in J^c} (\delta R_I^k)_{\{p,J\}}],
\end{equation}
where $c_{s,n}^k$ are constants to be specified below. Here 
\[
(\delta R_I^k)_{\{p,J\}} = R_{I,J}^k - \sum_{i=0}^s (-1)^i R_{I, \{p,J(\hat i)\}}^k,
\]
and the index set $\{p,J(\hat i)\}$ is given by 
\[
\{p,J(\hat i)\} = \{p, j_0, \ldots j_{i-1}, j_{i+1} , \ldots ,j_s\} \quad \text{if } J =  \{j_0, j_1, \ldots ,j_s\}.
\]
Alternatively, we can express the operator $A_{I,J}^k$ as 
\begin{equation}\label{def-Ak}
 A_{I,J}^k =  c_{s,n}^k \sum_{p \in J^c} \sum_{i=0}^s (-1)^{i} 
R_{I, \{p,J(\hat i)\}}^k.
\end{equation}
The constants $c_{n,s}^k$ are  given by
\begin{equation*}
c_{s,n}^k = \frac{(-1)^{1+ ks}(s!)^2}{(n-1) \cdots (n-s)}, \qquad 
1 \le s \le n-1, \quad \text{and } c_{0,n}^k = -1.
\end{equation*}
In fact, the key relations we will use below are that these constants satisfy 
\begin{equation}\label{c-relations}
 \frac{c_{s,n}^{k+1}}{c_{s,n}^{k}} = (-1)^{s} \quad \text{and } 
\frac{c_{s,n}^k}{c_{s-1,n}^k } = 
\frac{s^2}{(n-s)}(-1)^{k},
\end{equation}
which can be easily checked.
Note that when $s=k= 1$ and $n > 1$, we have $c_{1,n}^1 = 1/(n-1)$,
and as consequence, we see that \eqref{def-Ak} generalizes the
definition given in \eqref{A1Jij}.
Furthermore, $A_{I,J}^k = 0$ for $J \in \Gamma_s(I)$, $s >k$.

We next establish three key properties of the operator $A_{I,J}^k$,
using analogous properties established for the operator $R_{I,J}^k$.
\begin{lem}
\label{lem:Atrace}
The operators $A_{I,J}^k u$ only depend on the boundary traces of $u$
and satisfy for any $J \in \Gamma(I)$ and $j \in I^c$,
\begin{equation}
\label{trace-A}
\tr_{\lambda_j =0} A_{I^\prime,J}^k u = \tr_{\lambda_j =0} A_{I,J}^k u,
\end{equation}
where $I^\prime \in \Gamma$ is equal to $\{j, I \}$ up to a possible reordering. 
Furthermore, the operators $A_{I,J}^n$ are all identically zero.
 \end{lem}
\begin{proof}
For each $i \in J \subset
\Gamma_s(I)$, we have that $i \in I \cap \{p,J(\hat i)\}^c$. 
As a consequence, $A_{I,J}^k u$ only depends on $\tr_{\partial \S_n} u$.
The rest of the results follow
directly from the corresponding 
properties of the operators $R_{I,J}^k$ given in Lemma~\ref{lem:Rtrace}.
\end{proof}

\begin{lem}
\label{lem:Aprop1}
Let $I \in \Gamma$ and $J \in \Gamma_s(I)$. Then
the operator $\lambda_I^{-s}A_{I,J}^k u$ maps the spaces
$\Lambda^k(\partial \S_n) \to\Lambda^{k-s}(\S_n)$,
$\P_r\Lambda^k(\partial \S_n) \to \P_r \Lambda^{k-s}(\S_n)$,
and $\P_r^-\Lambda^k(\partial \S_n) \to\P_{r}^-\Lambda^{k-s}(\S_n)$.
\end{lem}
\begin{proof}
Since $A_{I,J}^k u$ is a linear combination of
operators of the form $R_{I, \{p,J(\hat i)\}} u$, 
for which $I \cap \{p,J(\hat i)\}^c$ is nonempty,
the lemma follows directly
from corresponding properties of the operators $R_{I,J}^k$, 
cf. Lemma~\ref{lem:Rprop1}.
\end{proof}
%

\begin{prop}
\label{prop:A-commute-prop}
Let $I \in \Gamma$ and $J \in \Gamma_s(I)$. The operators $A_{I,J}^k u$ satisfy 
the relations
\begin{equation}
\label{A-commute}
A_{I,J}^{k+1} du=  (-1)^{s}d(A_{I,J}^{k} u )
+ s (\delta A_I^k u)_{J},   0
 \le s \le k+1, \, 0 \le k \le n-1.
\end{equation}
\end{prop}
\begin{proof}
For $s=0$, the relation $dA_{I.J}^k u = A_{I,J}^k du$ follows from the
corresponding property of the pullbacks $P_{I,j}^*$.
Next we show that for 
for $s \ge 1$ we have 
\begin{equation}\label{delta-A}
(\delta A_I^ku)_{J} = - s c_{s-1,n}^k (\delta R_I^k u)_{J}. 
 \end{equation}
From \eqref{def-Ak2}, it follows that 
\[
(\delta A_I^ku)_{J}  =  c_{s-1,n}^k [(n-s+1) (\delta R_I^ku)_{J} - (\delta W)_{J} ],
\]
where $W_{J} = \sum_{p \in J^c} (\delta R_I^k u)_{\{p,J\}}$ for $I$ and $k$ fixed.
However, 
\begin{multline*}
(\delta W)_{J} = \sum_{a =0}^s (-1)^a [(\delta R_I^ku)_{\{a,J(\hat a)\}}
+ \sum_{p \in J^c} (\delta R_I^ku)_{\{p,J(\hat a)\}}]
\\
= (s+1) (\delta R_I^ku)_{J} 
+ \sum_{p \in J^c}  \sum_{a =0}^s (-1)^a R_{I,J(\hat a)}^k u
 -  \sum_{p \in J^c} ((\delta \circ \delta) R_{I,p}^ku)_J
\\
= (n+1)  (\delta R_I^ku)_{J},
\end{multline*}
where we have used the fact that for each fixed $p$,
\[
((\delta \circ \delta) R_{I,p}^ku)_J=
\sum_{a =0}^s (-1)^a \Big[ \sum_{i=0}^{a-1}(-1)^i R_{I,\{p,J(\hat a, \hat i) \}}   
- \sum_{i=a+1}^{s}(-1)^i R_{I,\{p,J(\hat a, \hat i) \}}\Big]  = 0.
\]
This implies \eqref{delta-A}.
For $s \ge 1$, it
follows from Proposition~\ref{prop:d-of-Q}, \eqref{c-relations}, \eqref{delta-A}, and the property $\delta \circ \delta = 0$ that
\begin{multline*}
d(A_{I,J}^k u) = c_{s,n}^k [(n-s) dR_{I,J}^k u - 
\sum_{p \in J^c} (\delta d R_I^k u)_{\{p,J\}}]\\
= c_{s,n}^k [ (n-s) R_{I,J}^{k+1} du -  \sum_{p \in J^c}(\delta R_I^{k+1} du)_{\{p,J\}}]
+ c_{s,n}^k (n-s) (-1)^{k-s} (\delta R_I^k u)_{J}\\
= \frac{c_{s,n}^k}{c_{s,n}^{k+1}} A_{I,J}^{k+1} du 
+ \frac{c_{s,n}^k(n-s)}{sc_{s-1,n}^k } (-1)^{k-s+1}  
(\delta A_I^k u)_{J}.
\\
= (-1)^s A_{I,J}^{k+1} du + s(-1)^{s+1}(\delta A_I^k u)_{J}.
\end{multline*}
This completes the proof.
\end{proof}

\section{The operator $E_n^k$}
In this final section, we prove that the operators $E_n^k$, defined by
\eqref{def-Tk}, are polynomial preserving cochain extensions.
\label{sec:Tnk}


\begin{thm}\label{thm:extension}
The operators $E_n^k$, defined by \eqref{def-Tk}, are extension operators.
\end{thm}
\begin{proof}
  We first note since the coefficients  $A_{I,J}^k u$ only depend on the boundary
  traces of $u$, the same is true of $E_n^ku$.  To establish the
  extension property, we will, without  loss of generality,  show that 
$\tr_{\lambda_0 =0} E_n^k u = \tr_{\lambda_0 =0} u.$ 
In fact, the primal operator $E_{n,0}^k$, given by 
\[
E_{n,0}^ku   = \frac{1}{n} \sum_{\substack{I \in \Gamma_m\\ 1 \le m \le n}} (-1)^{m+1} 
\sum_{j \in I}
\frac{\lambda_j}{\lambda_I} P_{I,j}^* u,
\]
is already an extension.
To see this we can argue as we have done above for scalar valued functions and one-forms.
The maps $P_{I,j}$, where $I \in \Gamma_1$ with $0,j \in I$ and $ j \neq 0$  
are projections onto $\{ x \, : \, \lambda_0(x) = 0 \}$. As a consequence, 
\[
\tr_{\lambda_0 =0}  \Big[\frac{1}{n} \sum_{\substack{I \in \Gamma_1\\ 0 \in I}}  
\sum_{j \in I}
\frac{\lambda_j}{\lambda_I} P_{I,i}^*\Big] u = \tr_{\lambda_0 =0} u,
\]
while the rest of the terms of $E_{n,0}^k u$ give no additional
contribution to the trace due to the cancellation of terms
corresponding to $I$ and $I^\prime = \{0,I \}$, where $0 \notin I$.
So to complete the proof, we need to show that
\[
\tr_{\lambda_0 =0} \Big[\frac{1}{n} \sum_{\substack{I \in \Gamma_m\\ 1 \le m\le n}} 
(-1)^{m+1}  \sum_{\substack{J \in \Gamma_s(I) \\ 1\le s \le k}}
\frac{\phi_J}{\lambda_I^{s+1}} \wedge A^k_{I,J} u \Big]  =0.
\] 
However, for terms corresponding to pairs $(I,J)$, with $0 \in J$, we
have $\tr_{\lambda_0=0} \phi_J= 0$, while if $0 \notin J$, the trace
of the terms corresponding to the pairs $(I,J)$ and $(\{0,I \},J)$
will cancel due to the trace property \eqref{trace-A} of the
coefficients.
\end{proof}

\begin{thm}\label{thm:cochain}
  The extensions $E_n^k$ are cochain maps, i.e., they satisfy
  $dE_n^{k} = E_n^{k+1} d$ for $0 \le k \le n-2$.  In
    addition, $d E_n^{n-1} u =0$.
\end{thm}
\begin{proof}
We first observe that for $I \in \Gamma$ and $J \in \Gamma_s(I)$, we have
\begin{equation}\label{dphi-over-lambda}
d\Big(\frac{\phi_J}{\lambda_I^{s+1}}\Big)  = (s +1) \sum_{i \in I \setminus J} 
\frac{\phi_{\{i,J\}}}{\lambda_I^{s+2}}.
\end{equation}
In particular, the right hand side is zero if $J=I$.
By the Leibniz rule, we have
\begin{multline*}
d E_n^ku  = \frac{1}{n} \sum_{\substack{I \in \Gamma_m\\ 1 \le m \le n}} (-1)^{m+1}
\sum_{\substack{J \in \Gamma_s(I) \\ 0 \le s \le k}}
 d \Big[\frac{\phi_J}{\lambda_I^{s+1}} \wedge  A_{I,J}^{k} u\Big] 
\\
= \frac{1}{n} \sum_{\substack{I \in \Gamma_m\\ 1 \le m \le n}} (-1)^{m+1}
\sum_{s=0}^{k} \sum_{J \in \Gamma_s(I)}
\Big[ d\Big(\frac{\phi_J}{\lambda_I^{s+1}} \Big)\wedge 
(A_{I,J}^{k} u)
+ (-1)^{s}\frac{\phi_J}{\lambda_I^{s+1}} d(A_{I,J}^k u)\Big].
\end{multline*}
However, by using \eqref{dphi-over-lambda}, we obtain for each fixed
$I \in \Gamma$,
\begin{align*}
\sum_{s=0}^{k}  \sum_{J \in \Gamma_s(I)}
d \Big(\frac{\phi_J}{\lambda_I^{s+1}} \Big)
\wedge (A_{I,J}^{k} u)
&= \sum_{s=0}^{k} (s+1) \sum_{J \subset \Gamma_s(I)}
\sum_{i \in I \setminus J}  
\frac{\phi_{i,J}}{\lambda_I^{s+2}}  \wedge (A_{I,J}^{k} u)
\\
&=\sum_{s=0}^{k} (s+1)\sum_{J \in \Gamma_{s+1}(I)}
 \frac{\phi_{J}}{\lambda_I^{s+2}} \wedge 
\sum_{i = 0}^{s+1} (-1)^i A_{I, J(\hat i)}^k u
\\
&=\sum_{s=0}^{k+1} s\sum_{J \in \Gamma_s(I)}
 \frac{\phi_{J}}{\lambda_I^{s+1}} \wedge (\delta A_I^k u)_{J}.
\end{align*}
Combining these results, noting that $A_{I,J}^k u =0$ for
$J \in \Gamma_{k+1}(I)$, and using \eqref{A-commute}, we get that
\begin{equation*}
d E_n^ku  = E_n^{k+1} du, \quad 0 \le k \le n-2, 
\quad \text{and } dE_n^{n-1} = 0.
\end{equation*}
This completes the proof.

\end{proof}

The final result we need to prove is that the extensions $E_n^k$
preserve smoothness and polynomial properties. The operator $E_n^k$
can be expressed as
\[
E_n^ku  = \frac{1}{n} \sum_{\substack{I \in \Gamma_m\\ 1 \le m\le n}} (-1)^{m+1} E_n^k(I) u,
\]
where each operator $E_n^k(I)$ is given by 
\[
E_n^k(I) = 
\sum_{\substack{J \in \Gamma_s(I) \\ 0 \le s \le k}}
\frac{\phi_J}{\lambda_I^{s+1}} \wedge A^k_{I,J} u.
\]
We will show below that each operator $E_n^k(I)$ preserves smoothness
and polynomial properties.  It is worth noting that it follows from
Lemma~\ref{lem:Aprop1} that the operators $E_n^k(I)$ map
$\Lambda^k(\partial \S_n)$ to $\lambda_I^{-1}\Lambda^k(\S_n)$, and
also
\begin{equation}\label{E-map-prop}
\P_r\Lambda^k(\partial \S_n) \to \lambda_I^{-1}\P_{r+1}\Lambda^k(\S_n), \quad \text{and }
\P_r^-\Lambda^k(\partial \S_n) \to \lambda_I^{-1}\P_{r+1}^-\Lambda^k(\S_n).
\end{equation}
In fact, to obtain the result for the trimmed spaces, we also need to
use the wedge product property for these spaces, for example expressed
by formula (3.16) of \cite{acta}.  To obtain preservation of
smoothness and polynomial properties, we need to remove the singular
factor $\lambda_I^{-1}$. The analysis below will lead to the following
fundamental result.

\begin{thm}\label{thm:pol-preserve}
The extension operator $E_n^k$ maps the spaces
$\Lambda^k(\partial \S_n) \to \Lambda^k(\S_n)$, 
$\P_r\Lambda^k(\partial \S_n) \to\P_r\Lambda^k(\S_n)$, and
$\P_r^-\Lambda^k(\partial \S_n) \to \P_r^-\Lambda^k(\S_n)$.
\end{thm}

The proof of this result will utilize the following 
alternative  representation of the operators $E_n^k(I)$.
\begin{lem}
\label{lem:Enkrep}
The operators $E_n^k(I)$ admit the representation 
\begin{equation}
\label{Enk-rep}
E_n^k(I)u = 
\frac{1}{m+1}
\Big[\sum_{j \in I} P_{I,j}^*u +d Q_{n}^k u+ Q_{n}^{k+1} du\Big],
\end{equation}
where the operator $Q_n^k = Q_n^k(I)$ is given by
\begin{equation*}
Q_n^k u= 
\sum_{\substack{J \in \Gamma_{s}(I) \\1\le s \le k}} 
\frac{1}{s} \frac{(\delta \phi)_J}{\lambda_I^s} \wedge  A_{I,J}^k u.
\end{equation*}
\end{lem}

Since the derivation of the alternative representation of the
operators $E_n^k$ is slightly technical, we will first use the
representation \eqref{Enk-rep} to prove
Theorem~\ref{thm:pol-preserve}.

\begin{proof}(of Theorem~\ref{thm:pol-preserve}) It is enough to show
  the desired mapping properties for each operator $E_n^k(I)$.  From
  the result of Lemma~\ref{lem:Aprop1}, it easily follows that the
  operator $Q_n^k(I)$ maps $\Lambda^{k}(\partial \S_n)$ to
  $\Lambda^{k-1}(\S_n)$ and $\P_r \Lambda^{k}(\partial \S_n)$ to
  $\P_{r+1} \Lambda^{k-1}(\S_n)$. By combining this with the fact that
  the operators $P_{I,j}$ are linear, the desired result in the smooth
  case and the full polynomial case follows.  Furthermore, since
  $\P_r^-\Lambda^k(\partial \S_n)$ is a subspace of
  $\P_r\Lambda^k(\partial \S_n)$, it follows from \eqref{E-map-prop}
  that $E_n^k(I)$ maps $\P_r^-\Lambda^k(\partial \S_n)$ into
\begin{equation}\label{pol-intersect}
\P_{r} \Lambda^{k}(\S_n) \cap  \lambda_I^{-1}\P_{r+1}^-\Lambda^k(\S_n).
\end{equation}
However, this space is identical to $\P_{r}^-\Lambda^k(\S_n)$. To see
this, let $u \in \P_{r+1}^-\Lambda^k(\S_n)$ such that
$\lambda_I^{-1}u \in \P_{r} \Lambda^{k}(\S_n)$.  It follows from the
definition of the trimmed spaces that for any $x_j \in \Delta_0(\S_n)$
\begin{equation*}
u \lrcorner (x-x_j) \in \P_{r+1}\Lambda^{k-1}(\S_n).
\end{equation*}
But since $\lambda_I^{-1}u$ is also in $\P_{r} \Lambda^{k}(\S_n)$
we have
\begin{equation*}
\lambda_I^{-1} (u \lrcorner (x-x_j)) = (\lambda_I^{-1}u) \lrcorner (x-x_j)
\in \P_{r+1}\Lambda^{k-1}(\S_n).
\end{equation*}
In other words, the polynomial form $u \lrcorner (x-x_j)$ has
$\lambda_I$ as a linear factor, and as a consequence,
$\lambda_I^{-1} (u \lrcorner (x-x_j))$ is also a polynomial form,
which then must be in $\P_{r}\Lambda^{k-1}(\S_n)$.  This implies that
$\lambda_I^{-1}u \in \P_{r}^- \Lambda^{k}(\S_n)$.  This shows that the
space given by \eqref{pol-intersect} is included in
$\P_{r}^-\Lambda^k(\S_n)$, and the opposite inclusion is
straightforward to check.
\end{proof}

It remains to establish the alternative representation \eqref{Enk-rep}
for the operators $E_n^k(I)$. Throughout the discussion below, the
index set $I \in \Gamma_m$, $1 \le m \le n$, will be considered
fixed. We first study the primal operator $E_{n,0}^k = E_{n,0}^k(I)$
given by
\[
E_{n,0}^ku  = 
\sum_{J \in \Gamma_0(I)}
\frac{\phi_J}{\lambda_I} \wedge A^k_{I,J} u 
= \sum_{i \in I}
\frac{\lambda_i}{\lambda_I} \wedge P_{I,i}^* u.
\]
In particular, for $k=0$ we have $E_{n,0}^0 = E_n^0(I)$.

\begin{lem}
\label{lem:En0rep}
For $I \in \Gamma_m$, the operator $E_{n,0}^k = E_{n,0}^k(I)$ has a
representation of the form
\begin{equation}
\label{En0rep}
E_{n,0}^k u =  \frac{1}{m+1} \Big[\sum_{j \in I} P_{I,j}^*u 
+ \lambda_I^{-1} \sum_{J \in \Gamma_1(I)} (\delta \phi)_J  
\wedge (\delta A_I^{k} u)_{J}\Big].
\end{equation}
\end{lem}
\begin{proof}
\begin{align*}
\lambda_I^{-1} \sum_{i \in I} \lambda_i P_{I,i}^*u &= \frac{1}{m+1} 
\sum_{j \in I} [P_{I,j}^*u + \lambda_I^{-1} \sum_{i \in I} \lambda_i 
(P_{I,i}^*u -P_{I,j}^*u )] \\ 
&=  \frac{1}{m+1} \Big[\sum_{j \in I} P_{I,j}^*u 
+  \lambda_I^{-1} \sum_{i \in I}  \sum_{\substack{j \in I\\ j < i}}
(\lambda_i - \lambda_j)(P_{I,i}^*u -P_{I,j}^*u )\Big]\\
&= \frac{1}{m+1}  \Big[\sum_{j \in I} P_{I,j}^*u +   \lambda_I^{-1} 
\sum_{J \in \Gamma_1(I)} (\delta \phi)_J \wedge (\delta A_I^ku)_{J} \Big].
\end{align*}
The representation for $E_{n,0}^k$ follows immediately.
\end{proof}

To establish an analogous result more generally, we will also need the
following preliminary result.
\begin{lem}
\label{lem:ddeltaphi}
For  $I \in \Gamma_m$ and $0 \le s \le m$, the following identity holds.
\begin{multline*}
\sum_{J \in \Gamma_s(I)} d \Big(\frac{(\delta \phi)_J}{\lambda_I^s} \Big) 
\wedge A_{I,J}^k u
+ \frac{s}{\lambda_I^{s+1} }\sum_{J \in \Gamma_{s+1}(I)} (\delta \phi)_J 
\wedge (\delta A_I^k u)_{J}
\\
= \frac{(m+1)s}{\lambda_I^{s+1}} \sum_{J \in \Gamma_s(I)} \phi_J \wedge A_{I,J}^k u.
\end{multline*}
\end{lem}
\begin{proof}
If $J \in \Gamma_{s-1}(I)$, we get from \eqref{dphi-over-lambda} that
\[
d\Big(\frac{\phi_J}{\lambda_I^s}\Big)  = s \sum_{p \in I \setminus J} 
\frac{\phi_{p,J}}{\lambda_I^{s+1}}.
\]
As a further consequence, we obtain that if $J \in \Gamma_s(I)$, then
\begin{align*}
d\Big(\frac{(\delta \phi)_J}{\lambda_I^s}\Big) 
&= \frac{s}{\lambda_I^{s+1}} \sum_{i =0}^s (-1)^{i}
\sum_{p \in I \setminus J(\hat i)}\phi_{p,J(\hat i)} \\
&= \frac{s(s+1)} {\lambda_I^{s+1}} \phi_J + \frac{s} {\lambda_I^{s+1}} 
\sum_{i =0}^s (-1)^{i}\sum_{p \in I \setminus J}\phi_{p,J(\hat i)}.
\end{align*}
If we wedge the second term with $A_{I,J}^k u$ and sum over
$J \in \Gamma_s(I)$, we obtain
\begin{multline*}
\sum_{J \in \Gamma_s(I)}\sum_{i =0}^s (-1)^{i}\sum_{p \in I \setminus J}
\phi_{p,J(\hat i)} \wedge A_{I,J}^k u
= - \sum_{J \in \Gamma_{s+1}(I)} \sum_{\substack{a,i =0 \\ a \neq i}}^{s+1} 
(-1)^{a+i} 
\phi_{J(\hat i)}\wedge A_{I,J(\hat a)} u 
\\
= - \sum_{J \in \Gamma_{s+1}(I)} (\delta \phi)_J \wedge (\delta A_I^ku)_{J}
+ (m-s)  \sum_{J \in \Gamma_{s}(I)} \phi_J \wedge A_{I,J}^k u.
\end{multline*}
The desired result follows by collecting terms. 
\end{proof}

To establish Lemma~\ref{lem:Enkrep}, we will make use of the following
operators,
\[
E_{n,\ell}^k    = E_{n,\ell}^k(I) = 
\sum_{\substack{J \in \Gamma_s(I)\\ 0 \le s \le \ell}}
\frac{\phi_J}{\lambda_I^{s+1}} \wedge A^k_{I,J} u, \quad 0 \le \ell \le k.
\]
In particular, we note that $E_{n,k}^k(I) = E_{n}^k(I)$.
\begin{proof} (of Lemma~\ref{lem:Enkrep})
We will prove by induction that the operator $E_{n,\ell}^k$ admits a
representation of the form
\begin{multline}
\label{Enlrep}
E_{n,\ell}^k u =  \frac{1}{m+1}\Big[\sum_{j \in I} P_{I,j}^*u + 
\sum_{J \in \Gamma_{\ell+1}(I)} \frac{(\delta \phi)_J}{\lambda_I^{\ell+1}}
\wedge (\delta A_I^k u)_{J}\\
+d Q_{n,\ell}^k  u+ Q_{n,\ell}^{k+1} du \Big],
\end{multline}
where,
\begin{equation*}
Q_{n,\ell}^k = 
\sum_{\substack{J \in \Gamma_{s}(I) \\1\le s \le \ell}} \frac{1}{s}
\frac{(\delta \phi)_J}{\lambda_I^s} \wedge  A_{I,J}^k u,
\end{equation*}
such that $Q_{n,k}^k = Q_n^k$ and $Q_{n,0}^k =0$.  Note that when
$\ell =0$, \eqref{Enlrep} is exactly the identity \eqref{En0rep},
while for $l=k$, $A_{I,J}^{k+1} du = (k+1) (\delta A^ku)_{I,J}$. This
gives
\begin{equation*}
Q_{n,k}^{k+1} du + 
\sum_{J \in \Gamma_{k+1}} \frac{(\delta \phi)_J}{\lambda_I^{k+1}}
\wedge (\delta A_I^k u)_{J} = Q_n^{k+1} du.
\end{equation*}
Therefore, the representation formula \eqref{Enk-rep} follows from
\eqref{Enlrep} for $l=k$.

It remains to carry out the induction argument to establish \eqref{Enlrep}.
If we assume that \eqref{Enlrep} holds for
$\ell-1$, then
\begin{multline*}
(m+1)E_{n,\ell}^k u  - \sum_{j \in I} P_{I,j}^* u
= dQ_{n,\ell -1}^k u + Q_{n,\ell -1}^{k+1} du \\
+ \sum_{J \in \Gamma_{\ell}(I)} \Big[(m+1)
 \frac{\phi_J}{\lambda_I^{\ell +1}} \wedge A_{I,J}^k u + 
\frac{(\delta \phi)_J}{\lambda_I^\ell} \wedge (\delta A_I^ku)_{J}\Big]
\\
= dQ_{n,\ell -1}^k u + Q_{n,\ell}^{k+1} du \\
+ \sum_{J \in \Gamma_{\ell}(I)} \Big[(m+1)
 \frac{\phi_J}{\lambda_I^{\ell +1}} \wedge A_{I,J}^k u +  
(-1)^{(\ell -1)} \frac{1}{\ell} \frac{(\delta \phi)_J}{\lambda_I^\ell}
\wedge dA_{I,J}^ku\Big],
\end{multline*}
where we have used the relation \eqref{A-commute}.
Now from Lemma~\ref{lem:ddeltaphi}, we get that 
\begin{multline*}
d(Q_{n,\ell}^k - Q_{n,\ell-1}^k) 
= \sum_{J \in \Gamma_{\ell}(I)}
\frac{1}{\ell} d\Big[\frac{(\delta \phi)_J}{\lambda_I^\ell} 
\wedge  A_{I,J}^k u\Big]
\\= 
\sum_{J \in \Gamma_{\ell}(I)} \frac{1}{\ell} 
\Big[(-1)^{\ell-1}\frac{(\delta \phi)_J}{\lambda_I^\ell} \wedge  d A_{I,J}^k u
+ (m+1)\ell \frac{\phi_J}{\lambda_I^{\ell+1}} \wedge  A_{I,J}^k u \Big]
\\
-  \sum_{J \in \Gamma_{\ell+1}(I)} 
 \frac{(\delta \phi)_J}{\lambda_I^{\ell +1}}
\wedge (\delta A_I^k u)_{J}.
\end{multline*}
Combining these results, we get
\begin{equation*}
(m+1)E_{n,\ell}^k u  -  \sum_{j \in I} P_{I,j}^* u
= d Q_{n,\ell}^k u + Q_{n,\ell}^{k+1}  du
+  \sum_{J \in \Gamma_{\ell+1}(I)} 
 \frac{(\delta \phi)_J}{\lambda_I^{\ell +1}}
\wedge (\delta A_I^k u)_{J},
\end{equation*}
which completes the induction argument, and the proof of the lemma.
\end{proof}

\section*{acknowledgement}The authors are grateful to Ralf Hiptmair
for pointing out the apparent similarity between the classical
rational blending method and constructions used in \cite{bubble-II}.

\bibliographystyle{amsplain}


\end{document}